\documentclass[11pt, a4paper]{article}
\usepackage[english]{babel}
\usepackage{amsmath}
\usepackage{amssymb}
\newtheorem{theorem}{Theorem}[section]
\newtheorem{lemma}[theorem]{Lemma}
\newtheorem{proposition}[theorem]{Proposition}
\newtheorem{remark}[theorem]{Remark}
\newenvironment{proof}[1][\proofname]{\noindent\textit{#1}.}{\hfill\raisebox{-1ex}{$\boxtimes$}}

\newcommand{\proofnameTwo}{Proof of Theorem~\ref{ML} modulo Proposition~\ref{ML_stable_ideals}}
\newcommand{\proofnameMain}{Proof of Proposition~\ref{ML_stable_ideals}}

\newcounter{tmpabcd}

\newcounter{tmpnum}

\newcounter{tmprome}

\newcommand{\mcc}{\mathbb{C}}

\newcommand{\mnn}{\mathbb{N}}

\newcommand{\mzz}{\mathbb{Z}}

\newcommand{\mqq}{\mathbb{Q}}

\newcommand{\diff}[1]{\frac{\partial}{\partial #1}}

\renewcommand{\b}[1]{{{#1}}}

\newcommand{\hidden}[1]{}

\newcommand{\idp}{\ensuremath{\mathcal{P}}}

\newcommand{\ul}[1]{\underline{#1}}
\newcommand{\ull}[1]{\underline{\b{#1}}}

\newcommand{\ord}{\ensuremath{{\rm ord}}}
\newcommand{\ordz}{\ensuremath{{\rm ord_{\b{z}=0}}}}

\newcommand{\trdeg}{\ensuremath{{\rm tr.deg.}}}

\begin{document} \sloppy{

%\centerline{}
%\begin{frontmatter}

\title{Multiplicity estimate for solutions of extended Ramanujan's system.}
%\shorttitle{Multiplicity estimate}

\author{Evgeniy \textsc{Zorin}\footnote{Institut de math\'ematiques de Jussieu, Universit\'e Paris 7, Paris, France. E-mail: evgeniyzorin@yandex.ru}}
%\abbrevauthor{ E. Zorin }
%\headabbrevauthor{Zorin, E.}

%\address{%
%\affilnum{1}Institut de math\'ematiques de Jussieu, Universit\'e Paris 7, Paris, France.}

% Address / e-mail address of corresponding author
%\correspdetails{evgeniyzorin@yandex.ru}

%\maketitle

%\date{}

%\author[authorlabel1]{Evgeniy ZORIN},
%\ead{evgeniyzorin@yandex.ru}
%\address[authorlabel1]{Institut de math\'ematiques de Jussieu, Universit\'e Paris 7, Paris, France.}
%\address[authorlabel2]{E-mail: evgeniyzorin@yandex.ru}
%\date{}

%\medskip
%\begin{center}
%{\small Received *****; accepted after revision +++++\\
%Presented by £££££}
%\end{center}

\maketitle

\begin{abstract}
\selectlanguage{english}
We establish a new \emph{multiplicity lemma} for solutions of a differential system extending Ramanujan's classical differential relations. This result can be useful in the study of arithmetic properties of values of Riemann zeta function at odd positive integers (Nesterenko, 2011). %Also, our result generalizes Nesterenko's multiplicity estimate (1996).
\end{abstract}

%Mathematics Subject Classification 2000: 11J81, 11J82, 11J61

%\end{frontmatter}

%This new multiplicity estimate opens the way to new results on algebraic independence and measures of algebraic independence, as it is shown in the author's Ph.D. thesis (2010)~\cite{EZ}.

%\tableofcontents

%\section{Introduction}

%\volumeyear{2011}
%\paperID{rnn999}

%\received{1 September 2011}
%\revised{11 September 2011}
%\accepted{21 September 2011}

%\communicated{A. Editor}

\selectlanguage{english}

\section{Introduction}

In what follows we denote by $\sigma_k(n)$, $k\in\mzz$, $n\in\mnn$ the sum of $k$th powers of divisors of $n$:
\begin{equation*}
    \sigma_{k}(n):=\sum_{d|n}d^k.
\end{equation*}
In this paper we consider the following sets of functions. First of all, the Eisenstein series
\begin{equation} \label{intro_functions_E}
    E_{2k}(z):=1-\frac{4k}{B_{2k}}\sum_{n=1}^{\infty}\sigma_{2k-1}(n)z^n,\quad k\in\mnn,
\end{equation}
where $B_{2k}$ are Bernoulli numbers. Also we consider
\begin{equation*}
    g_{u,v}(z):=\sum_{n=1}^{\infty}n^u\sigma_{-v}(n)z^n,\quad 0\leq u<v,\quad, u,v\in\mnn.
\end{equation*}

%In this paper we often use the following substitution: $q:=\ee^{2\pi i\tau}$. With this substitution functions $E_{2k}, g_{u,v}$ may be considered as functions of $\tau$. By a slight abuse of notation we shall denote $E_{2k}(\ee^{2\pi i\tau})$ as $E_{2k}(\tau)$ and $g_{u,v}(\e^{2\pi i\tau})$ as $g_{u,v}(\tau)$. This notation will not cause any ambiguity as the roles of $q$ and $\tau$ will never be interchanged.

It is well-known that functions $E_2$, $E_4$ and $E_6$ are algebraically independent over $\mcc(z)$ and all the other functions $E_{2k}$, $k\geq 4$ can be expressed in terms of $E_4$ and $E_6$ (see for instance~\cite{Serre1973}). More precisely, for all $k\geq 4$ there exists a polynomial $A_{k}\in\mcc[X,Y]$ such that
\begin{equation*}
    E_{2k}(z)=A_{k}(E_4(z),E_6(z)).
\end{equation*}
These polynomials $A_{k}(X,Y)$, $k\geq 4$ contain only monomials $M$ of bi-degrees $\left(\deg_X M,\deg_Y M\right)$ satisfying $2\deg_X M+3\deg_Y M=k$. %In particular, they satisfy $A_k(0,0)=0$.

In 2010 P.Kozlov %used this differential system to
proved (see~\cite{N2011}, page~2) that for any fixed $m\in\mnn$ all the functions
\begin{equation} \label{intro_ts}
    E_2(z),E_4(z),E_6(z),g_{u,v}(z),\quad 0\leq u<v\leq m
\end{equation}
are algebraically independent over $\mcc(z)$.

The functions~\eqref{intro_ts}
satisfy the following system of differential equations~\cite{N2011}. Denote $\delta:=z\frac{d}{dz}$. Then
\begin{equation} \label{intro_system_E}
    \delta E_2=\frac1{12}\left(E_2^2-E_4\right),\delta E_4=\frac1{3}\left(E_2E_4-E_6\right),\delta E_6=\frac12\left(E_2E_6-E_4^2\right)
\end{equation}
and for any odd $v\geq 3$
\begin{equation} \label{intro_system_g}
\begin{aligned}
    &\delta g_{u,v}(z)=g_{u+1,v}(z),\quad 0\leq u<v-1,\\&\delta g_{v-1,v}(z)=B_{2v+2}\frac{A_{v+1}(E_4(z),E_6(z))-1}{2v+2}.
\end{aligned}
\end{equation}
In the case $v=1$ one has
\begin{equation} \label{intro_system_g_1}
\delta g_{0,1}(z)=\frac{1}{24}\left(1-E_2(z)\right).
\end{equation}

Yu.Nesterenko~\cite{N2011} showed that values of functions $g_{u,v}(z)$ are closely related to the values of the Riemann zeta function $\zeta$ at odd positive integers. In particular, $\zeta(4n+3)\in\mqq(E_2(i),g_{0,4n+3}(i))$~\cite{N2011}. Whereas the system~\eqref{intro_system_E}, \eqref{intro_system_g}, \eqref{intro_system_g_1} for functions $E_2$, $E_4$, $E_6$, $(g_{u,v})_{0\leq u<v\leq m}$, $m\in\mnn$, is quite a simple extension of the system~\eqref{intro_system_E}, and in the case of the system~\eqref{intro_system_E} Nesterenko~\cite{N1996} established an optimal algebraic independence result for its solutions~\cite{N1996}, one may hope that this approach will lead to some results concerning algebraic independence of values of $\zeta$ at positive integral odd points. On this way, an important stage is \emph{a multiplicity lemma} for the functions in question.

In this paper we adopt the method from~\cite{N1996} and~\cite{NP}[Chapter~10] to establish (for any fixed odd $m\geq 3$) a multiplicity lemma for the whole set of functions $E_2$, $E_4$, $E_6$, $(g_{u,v})$, $0\leq u<v\leq m$, see Theorem~\ref{ML} below.

\section{Multiplicity Lemma}

Let $m\in\mnn$ be a fixed positive odd integer. We introduce the following notation:
\begin{equation*}
    R:=\mcc[X_0,X_1,X_2,X_3,Y_{0,1},Y_{0,3},Y_{1,3},Y_{2,3},\dots,Y_{m-1,m}].
\end{equation*}

\begin{theorem} \label{ML}
Let $m\geq 1$ be an odd integer. For all non-zero $P\in R$ there exists a constant $C$ depending on $m$ only such that
\begin{multline} \label{ML_result}
    \ord_{z=0}P(z,E_2(z),E_4(z),E_6(z),g_{0,1}(z),g_{0,3}(z),\dots,g_{0,m}(z),\dots,g_{m-1,m}(z))\\\leq C\left(\deg_{X_0} P+1\right)\left(\deg_{Eg} P+1\right)^{\left(\frac{m-1}2\right)^2+3},
\end{multline}
where $\deg_{Eg}P$ denotes the total degree of $P$ in the variables $X_1,X_2,X_3,Y_{0,1},\dots,Y_{m-1,m}$, i.e. all the variables appearing in $R$ but $X_0$.
\end{theorem}

\begin{remark} The exponent $\left(\frac{m-1}2\right)^2+3$ in the r.h.s. of~\eqref{ML_result} equals the number of functions different than $z$ in the l.h.s. of~\eqref{ML_result} and also the transcendence degree of $R$ over $\mcc(z)$. Hence Theorem~\ref{ML} provides multiplicity estimate with the optimal exponent.
\end{remark}

%\begin{notation}
In the sequel we denote
\begin{equation*}
D_0:=z\frac{d}{dz}+\frac{1}{12}\left(X_1^2-X_2\right)\frac{d}{dX_1}+\frac{1}{3}\left(X_1X_2-X_3\right)\frac{d}{dX_2}+\frac{1}{2}
    \left(X_1X_3-X_2^2\right)\frac{d}{dX_3},
\end{equation*}
\begin{equation*}
    D_1:=\frac{1}{24}\left(1-X_2\right)\frac{d}{dY_{0,1}},
\end{equation*}
\begin{equation*}
D_v:=\sum_{k=0}^{v-2}Y_{k+1,v}\frac{d}{dY_{k,v}}+B_{v+1}\frac{A_{v+1}(X_2,X_3)-1}{2v+2}\frac{d}{dY_{v-1,v}}, \quad v=3,5,\dots,m
\end{equation*}
and
\begin{equation} \label{intro_def_D}
    D:=D_0+\sum_{k=0}^{(m-1)/2}D_{2k+1}.
\end{equation}
The differential operator $D$ satisfies
\begin{multline}\label{sp_D}
    DP(z,E_2(z),E_4(z),E_6(z),g_{0,1}(z),\dots,g_{m-1,m}(z))\\=z\frac{d}{dz}P(z,E_2(z),E_4(z),E_6(z),g_{0,1}(z),\dots,g_{m-1,m}(z)).
\end{multline}

%\end{notation}

We deduce Theorem~\ref{ML} using Nesterenko's conditional Multiplicity Lemma (Theorem~1.1, Chapter~10~\cite{NP}). This result deals with differential system of the following type:
\begin{equation} \label{syst_diff}
f_i'(\b{z})=\frac{A_i(z,\ull{f})}{A_0(z,\ull{f})}, \quad i=1,\dots,n,
\end{equation}
where $A_i(\b{z},X_1,\dots,X_n)\in\mcc[z,X_1,\dots,X_n]$ for $i=0,...,n$ (we suppose that $A_0$ is a non-zero polynomial).
\begin{remark}
It is easy to see that system~(\eqref{intro_system_E},\eqref{intro_system_g},\eqref{intro_system_g_1}) is of the type~\eqref{syst_diff}.
\end{remark}
One associates to the system~\eqref{syst_diff} the differential operator
\begin{equation} \label{defD}
D_A = A_0(\b{z}, X_1,\dots, X_n)\diff{\b{z}} + \sum_{i=1}^nA_i(\b{z}, X_1,\dots, X_n)\diff{X_i}.
\end{equation}
In our case (i.e. the case of the system~\eqref{syst_diff}) this formula gives exactly the differential operator $D$ as defined in~\eqref{intro_def_D}.
\begin{theorem}[Nesterenko] \label{theoNesterenko_classique}
Suppose that functions
\begin{equation*}
\ull{f} = (f_1(\b{z}),\dots,f_n(\b{z})) \in \mcc[[\b{z}]]^n
\end{equation*}
are analytic at the point $\b{z}=0$ and form a solution of the system~(\ref{syst_diff}).
If there exists a constant $K_0$ such that every $D$-stable prime ideal $\idp \subset \mcc[X_1',X_1,\dots,X_n]$,
$\idp\ne(0)$, satisfies
\begin{equation} \label{ordIleqKdegI}
\min_{P \in \idp}\ordz P(\b{z},\ull{f}) \leq K_0,
\end{equation}
then there exists a constant $K_1>0$ such that for any polynomial $P \in
\mcc[X_1',X_1,\dots,X_n]$, $P\ne 0$, the following inequality holds
\begin{equation} %\label{LdMpolynome}
\ordz(P(\b{z},\ull{f})) \leq K_1(\deg_{\ul{X}'} P + 1)(\deg_{\ul{X}} P + 1)^n.
\end{equation}
\end{theorem}

To apply Theorem~\ref{theoNesterenko_classique} it is sufficient to prove Proposition~\ref{ML_stable_ideals} here below.

\begin{proposition} \label{ML_stable_ideals}
If $\idp$ is a prime ideal of
$$
R=\mcc[z,X_1,X_2,X_3,Y_{0,1},\dots,Y_{m-1,m}]
$$
with $D\idp\subset\idp$, then either $z\in\idp$ or $\Delta=X_2^3-X_3^2\in\idp$.
\end{proposition}

\begin{proof}[\proofnameTwo] If we have the result announced in Proposition~\ref{ML_stable_ideals}, then any prime $D$-stable ideal $\idp$ contains the polynomial
\begin{equation}\label{def_Theta}
    \Theta:=z\Delta=z\left(X_2^3-X_3^2\right).
\end{equation}
In this case we have obviously
\begin{multline*}
    \min_{P \in \idp}\ordz P(z,E_2(z),E_4(z),E_6(z),g_{0,1}(z),\dots,g_{m-1,m}(z))\\\leq\ordz\Theta(z,E_2(z),E_4(z),E_6(z),g_{0,1}(z),\dots,g_{m-1,m}(z)),
\end{multline*}
The quantity $K_0:=\ordz\Theta(z,E_2(z),E_4(z),E_6(z),g_{0,1}(z),\dots,g_{m-1,m}(z))$ is an absolute constant, in particular independent of $\idp$ (because $\Theta$ is just a concrete polynomial). Also, the quantity $K_0$ is finite, because all the functions $z,E_2(z),E_4(z),E_6(z),g_{0,1}(z),\dots,g_{m-1,m}(z)$ are algebraically independent over $\mcc$ and for this reason no polynomial vanishes on this set (i.e. in particular, $\Theta(z,E_2(z),E_4(z),E_6(z),g_{0,1}(z),\dots,g_{m-1,m}(z))$ is a non-zero function, analytic at $z=0$).
\end{proof}

To prove Proposition~\ref{ML_stable_ideals}, we describe at first principal $D$-stable ideals of $R$.

\begin{lemma} \label{lemma_principal}
There exists only two $D$-invariant principal prime ideals of $R$, namely, the ideals generated by $z$ and $\Delta$.
\end{lemma}
\begin{proof}
Suppose that $A\in R$ is any irreducible polynomial with the property that $A|DA$. Thus
\begin{equation}\label{eq_DA_AB}
    DA=AB,\quad B\in R.
\end{equation}
We readily verify with the definition of $D$ that $\deg_{\ul{Y}}DA\leq\deg_{\ul{Y}}A$ and $\deg_{z}DA\leq\deg_{z}A$, hence~\eqref{eq_DA_AB} implies $B\in\mcc[X_1,X_2,X_3]$.

%Further, with the definition of $D$ we verify that if $\deg_{Y_0}A>0$, then $\deg_{Y_0}DA<$

For any $F\in R$ we define the \emph{weight} of $F$ as
\begin{equation*}
    \phi(F):=\deg_tF(z,tX_1,t^2X_2,t^3X_3,t^{2m+2}\ul{Y}).
\end{equation*}
Then $\phi$ satisfies the following properties:
\begin{enumerate}
  \item For any $F\in R$
  \begin{equation*}
    \phi(DF)\leq \phi(F)+1.
  \end{equation*}
  \item For any $F,G\in R$
  \begin{equation*}
    \phi(FG)=\phi(F)+\phi(G).
  \end{equation*}
\end{enumerate}
 These properties together with~\eqref{eq_DA_AB} imply
  \begin{equation*}
    \phi(A)+\phi(B)=\phi(DA)\leq\phi(A)+1,
  \end{equation*}
hence $\phi(B)\leq 1$. Thus $B\in\mcc[X_1]$, $\deg B\leq 1$, i.e. $B=aX_1+b$, $a,b\in\mcc[z]$ and
\begin{equation} \label{eq_DA_lpA}
    DA=\left(aX_1+b\right)A.
\end{equation}
Also $\deg_zA+\deg_zB=\deg_zDA\leq\deg_zA$, hence $a,b\in\mcc$.

Now we consider another weight $\phi_2: R\rightarrow\mzz$. For any $F\in R$, we denote
\begin{equation*}
    \phi_2(F):=\deg_tF(z,tX_1,t^2X_2,t^3X_3,t^{-4}Y_{0,1},\dots,t^{-4m}Y_{0,m},t^{-4m+4}Y_{1,m},\dots,t^{-4}Y_{m-1,m})
\end{equation*}
(i.e. we assign to the variable $Y_{u,v}$ the weight $\phi_2(Y_{u,v}):=4(u-v)$).
Let $C$ be the sum of monomials of $A$ with minimal weight $\phi_2$. If we compare the sum of the monomials of weight $\phi_2(C)$ on both sides of~\eqref{eq_DA_lpA} and use the definition of $D$ we obtain
%\begin{equation} \label{eq_DC_bC}
%    z\frac{d}{dz}C=bC.
%\end{equation}
%In fact $\frac{d}{dY_{v-1}}C=0$, because otherwise the term $B_{v+1}\frac{-1}{2v+2}\frac{d}{dY_{v-1}}C$ would give a non-zero contribution to the l.h.s. of~\eqref{eq_DC_bC} of a weight strictly smaller than $\phi(C)$, and it is impossible as all the other terms in~\eqref{eq_DC_bC} have a weight $\geq\phi(C)$. Hence $C$ does not depend on $Y_{v-1}$. Further, $C$ does not depend on $Y_{v-2}$, because otherwise the term $Y_{v-1}\frac{d}{dY_{v-2}}C$ in the l.h.s. of~\eqref{eq_DC_bC} would imply that the r.h.s. of~\eqref{eq_DC_bC}, $bC$, (hence $C$) depends on $Y_{v-1}$ and we have just shown that it is not the case. Proceeding with recurrence we conclude that $C\in\mcc[z,X_1,X_2,X_3]$ and thus~\eqref{eq_DC_bC} has the form
\begin{equation} \label{eq_zC_bC}
    z\frac{d}{dz}C=bC
\end{equation}
(indeed, for any monomial $M$ and any differential operator $D_v$, $v=1,3,5,\dots,m$, all the non-zero monomials of $D_v(M)$ have weight $\phi_2$ strictly bigger than $\phi_2(M)$, also the only term in $D_0$ that does not increase $\phi_2$ is $z\frac{d}{dz}$, hence~\eqref{eq_zC_bC}).
Comparing the coefficients on the both sides of~\eqref{eq_zC_bC} we conclude $b=\deg_z C$, in particular $b\in\mzz$.

Substituting $X_1=E_2(z)$, $X_2=E_4(z)$, $X_3=E_6(z)$, $Y_{u,v}=g_{u,v}(z)$, $0\leq u<v\leq m$ in~\eqref{eq_DA_lpA} we obtain
\begin{multline}\label{eq_ordzDA_ordzlpA}
    \left(aE_2(z)+b\right)A\left(z,E_2(z),E_4(z),E_6(z),g_{0,1}(z),\dots,g_{m-1,m}(z)\right)\\=DA\left(z,E_2(z),E_4(z),E_6(z),g_{0,1}(z),\dots,g_{m-1,m}(z)\right).
\end{multline}
Let
\begin{equation*}
    A\left(z,E_2(z),E_4(z),E_6(z),g_{0,1}(z),\dots,g_{m-1,m}(z)\right)=cz^M+(\text{terms of order }>M),
\end{equation*}
$c\ne 0$, be the (first term of the) Taylor series of $A\left(z,E_2(z),E_4(z),E_6(z),g_{0,1}(z),\dots,g_{m-1,m}(z)\right)$. In view of the property~\ref{sp_D} we have
\begin{equation*}
    DA\left(z,E_2(z),E_4(z),E_6(z),g_{0,1}(z),\dots,g_{m-1,m}(z)\right)=cMz^M+(\text{terms of order }>M).
\end{equation*}
Using the Taylor series for $E_2$, \eqref{intro_functions_E}, notably the fact that $E_2(z)=1+\text{terms of order }>1$, we readily deduce from~\eqref{eq_ordzDA_ordzlpA}
\begin{equation*}
    (a+b)cz^M+(\text{terms of order }>M)=cMz^M+(\text{terms of order }>M).
\end{equation*}
Comparing coefficients with $z^M$ in the l.h.s. and in the r.h.s. of~\eqref{eq_ordzDA_ordzlpA} and simplifying out $c$  we readily deduce $a+b=M$. We have already established $b\in\mnn$. Obviously, $M\in\mnn$ (as it is a degree in a Taylor series). We conclude $a\in\mzz$.

So we have established that coefficients $a,b$ involved in~\eqref{eq_DA_lpA} are in fact integers.

Note that
\begin{equation} \label{eq_D_Delta_a_z_b}
    D(\Delta^{-a}z^{-b})=\left(-aX_1-b\right)\Delta^{-a}z^{-b}.
\end{equation}
We denote
\begin{equation}\label{def_S}
    S\left(z,E_2,E_4,E_6,g_{0,v},\dots,g_{v-1,v}\right):=A\left(z,E_2,E_4,E_6,g_{0,v},\dots,g_{v-1,v}\right)\Delta^{-a}z^{-b}.
\end{equation}
Applying the differential operator $D$ to the r.h.s. of~\eqref{def_S} and using \eqref{eq_DA_lpA}, \eqref{eq_D_Delta_a_z_b} we find out
\begin{equation*}
    DS=0.
\end{equation*}
Using~\eqref{sp_D} on the latter equality we conclude
\begin{equation*}
    \frac{d}{dz}S(z,E_2(z),E_4(z),E_6(z),g_{0,1}(z),\dots,g_{m-1,m}(z))=0,
\end{equation*}
hence
\begin{equation*}
    S(z,E_2(z),E_4(z),E_6(z),g_{0,1}(z),\dots,g_{m-1,m}(z))\in\mcc.
\end{equation*}
Recall that functions $z,E_2,E_4,E_6,g_{0,v},\dots,g_{v-1,v}$ are all algebraically independent over $\mcc$, see~\cite{N2011} page~2. For this reason we deduce $S[X_0,X_1,X_2,X_3,\ul{Y}]\in\mcc$ and thereby
\begin{equation*}
    A=\Delta^az^b.
\end{equation*}
If we suppose that $A$ is irreducible, we obtain immediately that either $(a,b)=(1,0)$ or $(a,b)=(0,1)$.
\end{proof}

\begin{proof}[\proofnameMain]
%To complete the proof of Proposition~\ref{ML_stable_ideals}
We consider the following nested sequence of rings
\begin{multline} \label{extension_tower}
    \mcc[z,\ul{X}]\subset\mcc[z,\ul{X},Y_{0,1}]\subset\mcc[z,\ul{X},Y_{0,1},Y_{2,3}]\subset\mcc[z,\ul{X},Y_{0,1},Y_{1,3},Y_{2,3}]
    \\\subset\mcc[z,\ul{X},Y_{0,1},Y_{0,3},Y_{1,3},Y_{2,3}]
    \subset\dots\subset\mcc[z,\ul{X},Y_{0,1},\dots,Y_{m-3,m-2}]\\\subset\mcc[z,\ul{X},Y_{0,1},\dots,Y_{m-3,m-2},Y_{m-1,m}]\\\subset
    \mcc[z,\ul{X},Y_{0,1},\dots,Y_{m-3,m-2},Y_{m-2,m},Y_{m-1,m}]
    \\\subset\dots\subset\mcc[z,\ul{X},Y_{0,m},\dots,Y_{m-1,m}]=R.
\end{multline}
We readily verify with the definition of $D$ that every term $R_{i}$ appearing in the chain~\eqref{extension_tower} satisfies $DR_i\subset R_i$.

Let $\idp\subset R$ be a prime ideal of $R$ satisfying $D\idp\subset\idp$. If $\idp\cap\mcc[z,\ul{X}]\ne \{0\}$, it contains a polynomial $z\Delta$ as shown in~\cite{Pellarin2006}[Theorem~1.4]. So everything is proved in this case. We suppose henceforth $\idp\cap\mcc[z,\ul{X}]=\{0\}$.

%If $\idp\cap\mcc[z,\ul{X},Y_{v-1}]\ne\{0\}$ and $\idp\cap\mcc[z,\ul{X}]=\{0\}$, then $\idp\cap\mcc[z,\ul{X},Y_{v-1}]$ is a principal ideal of $\mcc[z,\ul{X},Y_{v-1}]$, hence it contains $z\Delta$ by Lemma~\ref{lemma_principal} and claim of Proposition~\ref{ML_stable_ideals} is verified in this case. We can suppose in what follows $\idp\cap\mcc[z,\ul{X},Y_{v-1}]=\{0\}$.

%We proceed with recurrence. So either we find at some step an extension $K_i\subset K_{i+1}$ in the chain~\eqref{extension_tower} satisfying $\idp\cap K_i=\{0\}$ and $\idp\cap K_{i+1}\ne\{0\}$ and in this case we prove $z\Delta\in\idp$ with Lemma~\ref{lemma_principal}, either we obtain $\idp\cap R=\{0\}$.

We proceed with recurrence. As we suppose $\idp\ne\{0\}$ and $\idp\cap\mcc[z,\ul{X}]=\{0\}$, we find in the chain~\eqref{extension_tower} at some step an extension of rings $R_i\subset R_{i+1}$ satisfying $\idp\cap R_i=\{0\}$ and $\idp\cap R_{i+1}\ne\{0\}$. In this case the ideal (of the ring $R_{i+1}$) $\idp\cap R_{i+1}\ne\{0\}$ is a principal one, because we add exactly one variable at each step in the chain~\eqref{extension_tower}, i.e. $\trdeg_{R_i}R_{i+1}=1$. Hence $\idp\cap R_{i+1}$ is a $D$-stable principal ideal (of the ring $R_{i+1}$, and also this ideal generates a principal $D$-stable ideal of the ring $R$, because $D$-stability of a principal ideal means exactly the condition $Q|DQ$ on a generator of the ideal). We deduce with Lemma~\ref{lemma_principal} that $z\Delta\in\idp\cap R_{i+1}\subset\idp$, Q.E.D.
\end{proof}

%\section{Algebraic independence}

}

\end{document}